\documentclass[10pt,twocolumn,twoside]{IEEEtran}
\usepackage{amsmath,amssymb,amsthm}
\usepackage{mathtools}
\usepackage{graphicx}
\usepackage{xcolor}
\usepackage{enumitem}
\usepackage{tikz}
\usetikzlibrary{decorations.pathreplacing}
\usetikzlibrary{positioning}

\usepackage{datetime}
\usepackage{nicematrix}

\definecolor{LucaColor}{RGB}{0,90,180}
\definecolor{CinziaColor}{RGB}{180,60,0}
\definecolor{AlessandroColor}{RGB}{120,0,160}

\newtheorem{theorem}{Theorem}

\theoremstyle{remark}
\newtheorem{remark}{Remark}
\theoremstyle{definition}
\newtheorem{definition}{Definition}
\newtheorem{example}[theorem]{Example}

\newcommand{\rank}{\operatorname{rank}}
\newcommand{\R}{\mathbb{R}}
\newcommand{\C}{\mathbb{C}}
\newcommand{\Rn}{\R^n}

\newcommand{\RNN}{\R^{N\times N}}
\newcommand{\ones}{\mathbf{1}}
\def \S {\mathcal S}

\renewcommand{\vec}[1]{{\boldsymbol{#1}}}
\newcommand{\bmat}[1]{\begin{bmatrix}#1\end{bmatrix}}
\newcommand{\smat}[1]{ \left[\begin{smallmatrix} #1 \end{smallmatrix}\right]}
\def \x {{\vec{x}}}
\def \f {{\vec{f}}}
\def \F {{\vec{F}}}

\def \z {{\vec{z}}}
\def \e {{\vec{e}}}
\def \y {{\vec{y}}}
\def \v {{\vec{v}}}

\newcommand{\re}{\operatorname{Re}}

\newcommand{\svdots}{%
  \vbox{%
    \baselineskip=1.5pt
    \lineskiplimit=0pt
    \kern1.5pt
    \hbox{$\scriptscriptstyle .$}%
    \hbox{$\scriptscriptstyle .$}%
    \hbox{$\scriptscriptstyle .$}%
  }%
}

\newcommand{\sdots}{{\scriptscriptstyle\dots}}

\newcommand{\sddots}{%
  \mathinner{%
    \raise3pt\vbox{\hbox{$\scriptscriptstyle .$}}%
    \raise1.5pt\hbox{$\scriptscriptstyle .$}%
    \hbox{$\scriptscriptstyle .$}%
  }%
}

\title{On the Optimal Laplacian Jordan Structure for Synchronizability}
\author{Luca Dieci, Cinzia Elia, and Alessandro Pugliese%
\thanks{Corresponding author: Luca Dieci.
L. Dieci is with the School of Mathematics, Georgia Institute of
Technology, Atlanta, GA 30332, USA (e-mail: dieci@math.gatech.edu).
C. Elia and A. Pugliese are with the Department of Mathematics,
University of Bari Aldo Moro, 70125 Bari, Italy (e-mail:
cinzia.elia@uniba.it; alessandro.pugliese@uniba.it).}}

\begin{document}

\maketitle


\begin{abstract}
In this work, for a network of differential equations with a prescribed graph structure, our goal is to show how to select the Laplacian of the network in order to obtain the most favorable outcome insofar as {\sl synchronizability}.  That is, we will want to: (i) guarantee asymptotic stability of a synchronous solution (as measured by a negative value of the master stability function), (ii) minimize the normalized spread of the Laplacian eigenvalues, and (iii) have a transient as short as possible.  Within the class of tridiagonal Laplacians, we give both necessary and sufficient conditions for satisfying our three criteria above, and give extension to banded Laplacians as well.  Finally, we give extensive numerical results to elucidate our theoretical results and to compare to existing works.
\end{abstract}
\thanks{This work has been partially supported by the Georgia Institute of Technology and the University of Bari Aldo Moro.}

\section{Introduction}\label{intro}
In this work we consider a network of $N$ agents, where each agent separately obeys the same differential equation, and the agents are coupled through linear diffusive coupling:
\begin{equation}\label{Network}
\begin{aligned}
\dot \x_i &= \f(\x_i)+\sum_{j=1}^N a_{ij}E(\x_j-\x_i),\\
&i=1,\dots,N,\qquad t\ge0.
\end{aligned}
\end{equation}
Here, $\x_i\in \Rn$, $A=(a_{ij})_{i,j=1, \ldots, N}$, is a weighted
adjacency matrix of a rooted directed graph (i.e., it has a spanning tree), and so $a_{ij}\ge 0$, $a_{ij}=0$ if node (agent) $i$ and $j$ are not connected, and further $a_{ii}=0$.  In \eqref{Network}, $E \in \R^{n \times n}$ is a matrix of $0$'s and $1$'s specifying which components interact with one another. \\
Let $L=\left(l_{ij}\right)_{i,j=1}^N \in \R^{N\times N}$ be the {\sl generalized Laplacian} matrix of the graph: $l_{ij}=-a_{ij}$, $i\ne j$, and $l_{ii}=-\sum_j l_{ij}$,
$i,j=1,\dots, N$.  Note that $L$ is not necessarily symmetric, but it has $0$ row-sum, hence it is singular, and since the diagonal entries of $L$ are nonnegative then the non-zero eigenvalues of $L$ have positive real parts, as an immediate consequence of Gerschgorin circle theorem.  Moreover, since the graph is a rooted digraph, $0$ is a simple eigenvalue and the eigenvalues of $L$  -not necessarily distinct- satisfy
$\lambda_1=0<\re(\lambda_i)$, $i=2,\dots, N$. \\
Now, let $\x=\begin{bmatrix} \x_1 \\ \vdots \\ \x_N \end{bmatrix}\in \R^{nN}$, 
$\F(\x)=\begin{bmatrix} \f(\x_1) \\ \vdots \\ \f(\x_N) \end{bmatrix}$, 
and rewrite \eqref{Network} as
\begin{equation}\label{Network2}
	\dot \x = \F(\x)-B \x ,\quad\text{where}\quad
	B=L\otimes E.	
\end{equation}
We note that if $\bar \x: \ t\ge 0 \mapsto \bar \x(t)\in \Rn$ is a solution of 
$\dot \z=\f(\z)$, then 
\eqref{Network2} always has the
so called {\sl synchronous solution} $\x_s=(\bar \x, \ldots, \bar \x)$.
However even if $\bar \x$ is (exponentially) asymptotically stable for the single agent, there is no guarantee of asymptotic stability of $\x_s$ for the full network.  A useful tool to ascertain this fact is the so-called {\sl Master Stability Function} (MSF), introduced with this name 
by Pecora and Carroll in \cite{PecoraCarroll}.
\subsection{The MSF}
Define the {\emph{synchronous subspace}} $\S$ by
\begin{equation}\label{SynchSub}
\begin{aligned}
\S=\{\x\in \R^{nN}:\;&\x_1=\x_2=\cdots=\x_N,\\
&\x_j\in\Rn,\quad j=1,\dots,N\}.
\end{aligned}
\end{equation}
Assume that the isolated agent has a local attractor, and let $\bar{x}$ be a typical trajectory on it. We further assume that the Lyapunov exponents along $\bar{x}$ are stable, 
a sufficient condition for which is that each Lyapunov exponent belongs to a distinct interval of the exponential dichotomy spectrum (see \cite{DV2002}), which describes the possible uniform exponential growth and decay rates of the linearized system.
  
Solutions of \eqref{Network2} in $\S$ in a neighborhood of $\x_s=(\bar \x, \ldots, \bar \x)$, converge to $N$ copies of a solution on the local attractor,
but this might not be the case for solutions that are transversal to $\S$. In order to study local stability of $\x_s$, we linearize \eqref{Network2} along $\x_s$ and obtain the following linear system 
\begin{equation*}
\dot\y=\bigl(I_N\otimes D\f(\bar\x)-B\bigr)\y.
\end{equation*}
where $\y=\smat{\y_1 \\ \vdots \\ \y_N}$. Let $V\in\C^{N\times N}$ give the Jordan form of $L$, so that $L=VJV^{-1}$, with
\begin{equation*}
J=\operatorname{diag}\bigl(0,J_2(\lambda_2),\ldots,J_p(\lambda_p)\bigr),
\end{equation*}
and
\begin{equation*}
J_i(\lambda_i)=\lambda_i I_{N_i}+J_{N_i}(0),
\qquad i=2,\ldots,p,
\end{equation*}
where $J_{N_i}(0)$ is the nilpotent Jordan block of size $N_i$ and $1+N_2+\cdots+N_p=N$.
Performing the change of variable $(V^{-1}\otimes I_n) \y \to \z$, one obtains the linear system 
\begin{equation}\label{Jordan1}
\dot \z=(I_N \otimes D\f(\bar \x) -J \otimes E) \z\ .
\end{equation}
Now, partition 
$$\z=\begin{bNiceMatrix}[first-row, first-col] \\
{\scriptstyle n}\{ & \z_1 \\ {\scriptstyle nN_2}\{ & \z_2 \\ & \vdots \\ {\scriptstyle nN_p}\{ & \z_p 
\end{bNiceMatrix}\,,$$
and so from \eqref{Jordan1} one gets $\dot \z_1=D\f(\bar \x)\z_1$ and 
\begin{equation}\label{Jordan}
\begin{aligned}
\dot \z_i={}&\bigl(I_{N_i}\otimes D\f(\bar\x)
-J_i(\lambda_i)\otimes E\bigr)\z_i,\\
&i=2,\ldots,p.
\end{aligned}
\end{equation}
In order to study the asymptotic stability of $\x_s$, from \eqref{Jordan1} we want to study the asymptotic stability of the origin for $\z$.  Hence, aside from the concerns of stability of $\bar \x$ for the single agent (hence of the origin in $\R^n$ for $\z_1$), 
one needs to infer asymptotic stability of the origin in $\R^{nN_i}$ for the $\z_i$'s in \eqref{Jordan}, $i=2,\dots, p$.  To this end, given the form of $J_i(\lambda_i)$, rewriting
$$\z_i=\begin{bNiceMatrix}[first-col]
{\scriptstyle n}\{ & \z_{i,1} \\
{\scriptstyle n}\{ & \z_{i,2} \\
                   & \vdots   \\
{\scriptstyle n}\{ & \z_{i,N_i},
\end{bNiceMatrix},\ i=2,\dots, p\,,$$ it is sufficient to look at
$$\dot \z_{i,N_i} = (D\f(\bar\x) - \lambda_i E) \z_{i,N_i}\ .$$
As a consequence, one can consider the single general parametrized linear system
\begin{equation}\label{Linear}
\dot \v = (D\f(\bar\x(t)) -\eta E)\v\ , \quad \eta \in \C,\, \re(\eta)> 0 .
\end{equation}
Denote with $\Lambda(\eta)$ the largest Lyapunov exponent of \eqref{Linear}: this is the Master Stability Function (MSF) of \cite{PecoraCarroll}. Let 
\begin{equation}\label{SynchR}
\mathcal R=\{ \eta \in \C,\, \re(\eta)> 0\,  | \ \Lambda(\eta) <0\}. 
\end{equation}
We call $\mathcal R$ the \emph{synchronization region}.  If $\lambda_i \in \mathcal R$ for all $i=2, \ldots, p$, then we say that $\x_s$ is transversally asymptotically stable for \eqref{Network2}. 

With this, we can pinpoint the first criterion for synchronizability: if $\mathcal R$ is not empty, the eigenvalues of $L$ must lead to a negative value of the MSF.  This has been recognized for a long time; e.g., see \cite{PecoraCarroll, Arenas, Universality}.  In particular, one may attempt choosing the graph structure (topology of the network) to force this feature, as done in \cite{TopoView, OptimTopo}.
Alternatively, if the graph structure is not negotiable, say it has to be that of nearest-neighbor interaction with all active arcs (tridiagonal, unreduced, Laplacian), then one must assign weights to the arcs in order to guarantee that $L$ has eigenvalues leading to a negative MSF, as we did in \cite{DiElPu1}; this point of view leads to non-symmetric (in particular, non-normal) Laplacians. 
\begin{remark}
An alternative to local stability theory based on the MSF can be given using global properties of the vector field $\f$ of the single agent. This point of view was adopted in \cite{DeLellis} 
for \eqref{Network} with $E=I$ using the QUAD property and leading to results guaranteeing synchronization. 
\end{remark}
\subsection{The transient}\label{Transient}
Although the asymptotic stability of \eqref{Jordan} is not affected by the size of the Jordan blocks $J_i$'s, its transient behavior is heavily affected by it, as we now show.  \\
Let $\lambda_i$ 
be such that the greatest Lyapunov exponent of \eqref{Linear} for $\eta=\lambda_i$ is negative. 
Write explicitly the ODE for the $\z_i$ in \eqref{Jordan}: 
\begin{equation}\label{SingleJblock}
	\left \{ \begin{array}{cc}
		\dot \z_{i,1}= & A_i(t)\z_{i,1} -E \z_{i,2} \\
		\dot \z_{i,2}= & A_i(t)\z_{i,2} -E \z_{i,3} \\
		\vdots  &  \vdots \\
		\dot \z_{i,N_i-1}= & A_i(t)\z_{i,N_i-1} -E \z_{i,N_i} \\
		\dot \z_{i,N_i}= & A_i(t)\z_{i,N_i} 
	\end{array} \right . ,
\end{equation}
where we have set $A_i(t)=(D\f(\bar\x(t))-\lambda_i E)$. Let $Z_i(t)$ be the principal matrix solution of 
$\dot Z_i=A_i(t)Z_i$, $Z_i(t) \in \C^{n \times n}$ 
and let $Z_i(t,s)=Z_i(t)Z_i(s)^{-1}$, $t \geq s \geq 0$. 
Since $\lambda_i\in \mathcal {R}$, then there exist $K>0$ and $\alpha>0$ such that $\|Z_i(t)\| \leq Ke^{-\alpha t}$, for $t \geq 0$, and in particular $\|\z_{i,N_i}\|\le Ke^{-\alpha t}\|\z_{i,N_i}(0)\|$. Assume moreover that the exponential dichotomy spectrum is contained in the negative real half-line, so that $Z_i(t)$ is uniformly asymptotically stable, i.e., $\|Z(t,s)\| \leq Ke^{-\alpha (t-s)}$, for $t \geq s \geq 0$. 
Using the variation of constants formula, we obtain 
\begin{equation*}
\begin{aligned}
\z_{i,N_i-1}(t)={}&Z_i(t)\z_{i,N_i-1}(0)\\
&-\int_0^t Z_i(t,s)E Z_i(s)\z_{i,N_i}(0)\,ds.
\end{aligned}
\end{equation*}
Now, we can bound $\|\z_{i,N_i-1}(t)\|$ as
\begin{equation*}
\begin{aligned}
\|\z_{i,N_i-1}(t)\|\le e^{-\alpha t}\bigl(&K\|\z_{i,N_i-1}(0)\|\\
&+K^2\|E\|\|\z_{i,N_i}(0)\|t\bigr).
\end{aligned}
\end{equation*}
Analogously, we obtain the bounds
\begin{equation*}
\begin{aligned}
\|\z_{i,N_i-2}(t)\|\le e^{-\alpha t}\biggl(&K\|\z_{i,N_i-2}(0)\|\\
&+K^2\|E\|\|\z_{i,N_i-1}(0)\|t\\
&+K^3\|E\|^2\|\z_{i,N_i}(0)\|\frac{t^2}{2}\biggr).
\end{aligned}
\end{equation*}
and eventually
\begin{equation}\label{TransGrowth}
\|\z_{i,1}(t)\|\le e^{-\alpha t}\sum_{j=0}^{N_i-1}K^{j+1}\|E\|^j
\|\z_{i,j+1}(0)\|\frac{t^j}{j!}.
\end{equation} 
It follows from \eqref{TransGrowth} that the transient behavior of $\z_i(t)$ is impacted by the dimension $N_i$ of $J_i(\lambda_i)$ through the dependency on the factor 
$\displaystyle{e^{-\alpha t}\sum_{j=1}^{N_i-1}K^{j+1} \|E\|^j \|\z_{i,j+1}(0)\| \frac{t^j}{j!}}$.  In Section 
\ref{LorSim}, we exemplify the practical impact of the size of the Jordan blocks on a network of coupled Lorenz systems.

Next, in Section \ref{TridLap} we show how to choose the Jordan form of the Laplacian (eigenvalues, and size of the Jordan blocks), for the case of ``tridiagonal'' Laplacians (nearest neighbor graphs), and in Section \ref{BandLap} we consider the case of banded Laplacians, in order to obtain as short a transient as possbile.
Finally, in Section \ref{NumSim} we highlight the practical impact of the above result by comparing our classes of Laplacians with other choices used in the literature.  We should stress that 
our results are in the spirit of our previous effort \cite{DiElPu1}, that is when we have a ``non-negotiable'' network structure, say tridiagonal or banded, as this is the case of most practical interest to us.   It would be a totally different story to consider having a preassigned number of arcs (but not their location) and assign them between nodes so to enhance synchronizability, as recently discussed in \cite{NM2010, LuUrschelLiu}.

\section{Tridiagonal Laplacians with $\lambda_2=\ldots=\lambda_N=1$}\label{TridLap}

A graph Laplacian matrix $L\in\R^{N\times N}$ is called a tridiagonal Laplacian matrix if it has the form
\begin{equation}\label{eq:laplacian}
\begin{gathered}
L=\bmat{
a_1 & b_2 & 0 & \cdots & 0\\
c_2 & a_2 & b_3 & \ddots & \vdots\\
0 & c_3 & a_3 & \ddots & 0\\
\vdots & \ddots & \ddots & \ddots & b_N\\
0 & \cdots & 0 & c_N & a_N},\\
b_i\le0,\quad c_i\le0,\quad c_i+a_i+b_{i+1}=0,\\
i=1,\dots,N.
\end{gathered}
\end{equation}
where we set $c_1=b_{N+1}=0$.

We already know that $L$ has a simple $0$ eigenvalue and all other eigenvalues of $L$ have nonnegative real part. Now we claim that all eigenvalues of $L$ are real. Indeed, recall that $L$ is irreducible if and only if
$$
b_i<0\quad\text{and}\quad c_i<0\qquad \text{for all } i=2,\dots,N.
$$
Since an irreducible tridiagonal matrix is diagonally similar to a real symmetric, irreducible, tridiagonal matrix, then an irreducible tridiagonal matrix $L$ has real and distinct eigenvalues. By continuity, it follows that any graph Laplacian $L$ as in \eqref{eq:laplacian} has, possibly repeated, nonnegative real eigenvalues: $0=\lambda_1< \lambda_2\le \cdots\le \lambda_N$.

In \cite{NM2010}, the authors introduce the normalized spread of the eigenvalues
\begin{equation}\label{eq:normalized_spread_eigenvalues}
\sigma^2\coloneqq \frac{1}{d^2(N-1)}\sum_{i=2}^N |\lambda_i-\bar{\lambda}|^2
\end{equation}
as a measure of synchronizability. Here $d$ is the average coupling strength per node, which in our notation, for \eqref{eq:laplacian}, is given by
\begin{equation}\label{eq:average_coupling_strength}
d\coloneqq -\frac{1}{N}\sum_{i=1}^N \sum_{\substack{j=1\\ j\neq i}}^N L_{ij}
=\frac{1}{N}\sum_{i=1}^N a_i,
\end{equation}
and $\bar{\lambda}$ is the average of the nonzero eigenvalues, namely
$$
\bar{\lambda}\coloneqq \frac{1}{N-1}\sum_{i=2}^N \lambda_i.
$$
We are interested in tridiagonal Laplacian matrices that minimize $\sigma^2$. Clearly, $\sigma^2\ge 0$, and $\sigma^2=0$ if and only if
\begin{equation}\label{MSpectrum}
\lambda_2=\cdots=\lambda_N.
\end{equation}
An equivalent measure of synchronizability which has often been advocated (e.g., see \cite{Ratio}), for Laplacians with real eigenvalues, is  to use the ratio $\lambda_N/\lambda_2$ of the largest to smallest (nonzero) eigenvalues of the Laplacian, arguing that the smaller is the ratio the more synchronizable is the network; obviously, this ratio is minimized precisely when all eigenvalues are equal, thus leading us once more to the condition \eqref{MSpectrum}.
Therefore, we will focus on tridiagonal Laplacian matrices that satisfy \eqref{MSpectrum} and since $\sigma^2$ is invariant under multiplication of $L$ by a positive scalar, we will focus on matrices satisfying
\begin{equation}\label{eq:eigenvalues_all_ones}
\lambda_2=\cdots=\lambda_N=1, 
\end{equation}
that is on matrices that have $0$ as a simple eigenvalue and $1$ as an eigenvalue of algebraic multiplicity $N-1$. We now give a complete classification of the tridiagonal Laplacian matrices $L$ that satisfy \eqref{eq:eigenvalues_all_ones} and we note that the structure of $S_{k,N}$ in \eqref{eq:SkN} below includes the structure envisioned in \cite[Theorem 7]{NM2006}.


\begin{theorem}\label{thm:classification_trid}
Let $L\in\R^{N\times N}$ be a tridiagonal Laplacian matrix as in \eqref{eq:laplacian}. Then $L$ has spectrum $\{0,1,\dots,1\}$ if and only if it has the form
\begin{equation}\label{eq:SkN}
\resizebox{0.94\columnwidth}{!}{$\displaystyle
L=S_{k,N}(t)\coloneqq
\bmat{
1 & -1 & 0 & 0 & 0 & 0 & \cdots & 0 & 0 \\
0 & \ddots & \ddots & 0 & 0 & 0 & \cdots & 0 & 0 \\
0 & 0 & 1 & -1 & 0 & 0 & \cdots & 0 & 0 \\
0 & 0 & 0 & t & -t & 0 & \cdots & 0 & 0 \\
0 & 0 & 0 & t-1 & 1-t & 0 & \cdots & 0 & 0 \\
0 & 0 & 0 & 0 & -1 & 1 & \cdots & 0 & 0 \\
\vdots & \vdots & \vdots & \vdots & \vdots & \vdots & \ddots & \ddots & \vdots \\
0 & 0 & 0 & 0 & 0 & 0 & \cdots & -1 & 1
}. $}
\end{equation}
Here $k\in\{1,\dots,N-1\}$, and
\begin{equation}\label{eq:core}
\bmat{
 t & -t \\
 t-1 & 1-t
}, \text{ for some } t\in[0,1],
\end{equation}
is the principal submatrix of $S_{k,N}(t)$ corresponding to rows and columns $k$ and $k+1$. 
\end{theorem}

\begin{proof}
A matrix $L$ as in \eqref{eq:laplacian} with eigenvalues $\{0,1,\dots,1\}$, clearly satisfies
$$
\operatorname{tr}(L)=\operatorname{tr}(L^2)=N-1.
$$
Now, it is straightforward to see that: 
\begin{equation*}
\begin{aligned}
\operatorname{tr}(L)
&=\sum_{i=1}^N a_i
=-\sum_{i=2}^{N}(b_i+c_i),\\
\operatorname{tr}(L^2)
&=\sum_{i=1}^N a_i^2+2\sum_{i=2}^N b_i c_i.
\end{aligned}
\end{equation*}

Using the zero row-sum condition from \eqref{eq:laplacian}, we have:
\begin{equation*}
\begin{aligned}
\sum_{i=1}^N a_i^2
={}&b_2^2+\sum_{i=2}^{N-1}(c_i+b_{i+1})^2+c_N^2\\
={}&\sum_{i=2}^{N}(c_i^2+b_i^2)
+2\sum_{i=2}^{N-1}c_i b_{i+1}.
\end{aligned}
\end{equation*}
and therefore:
\begin{equation*}
\begin{aligned}
\operatorname{tr}(L^2)={}&\sum_{i=2}^{N}(c_i^2+b_i^2)
+2\sum_{i=2}^{N-1}c_i b_{i+1}+2\sum_{i=2}^N b_i c_i.
\end{aligned}
\end{equation*}
Now note that:
\begin{flushleft}
\(\displaystyle
\sum_{i=2}^{N}(c_i+b_i+1)^2
=\sum_{i=2}^{N}(c_i^2+b_i^2)+N-1+2\sum_{i=2}^{N}(b_i+c_i)
\)\\[-0mm]
\(\displaystyle
{}+2\sum_{i=2}^{N}b_i c_i
=\operatorname{tr}(L^2)-2\sum_{i=2}^{N-1}c_i b_{i+1}+N-1-2\operatorname{tr}(L)
\)\\[-0mm]
\(\displaystyle
{}=-2\sum_{i=2}^{N-1}c_i b_{i+1}.
\)
\end{flushleft}
That is:
$$
\sum_{i=2}^{N}(c_i+b_i+1)^2+2\sum_{i=2}^{N-1}c_i b_{i+1}=0.
$$
Every term of the sum on the left-hand side of this last expression is nonnegative, so each term must be zero. Hence we have
\begin{equation}\label{eq: relation 1}
c_i+b_i=-1 \quad i=2,\dots,N, 
\end{equation}
and
\begin{equation}\label{eq: relation 2}
c_i b_{i+1}=0\quad i=2,\dots,N-1. 
\end{equation}

Using \eqref{eq: relation 1} and \eqref{eq: relation 2} above, we see that:
\begin{enumerate}
 \item If $b_k=0$, with $2\le k\le N$, then
\begin{equation*}
c_i=-1,\quad a_i=1,\quad b_{i+1}=0,\quad i=k,\ldots,N.
\end{equation*}
 \item If $c_k=0$ with $2\le k\le N$, then
\begin{equation*}
c_i=0,\quad a_i=1,\quad b_{i+1}=-1,\quad i=1,\ldots,k-1.\\
\end{equation*}
 \item If $b_kc_k\ne0$, with $2\le k\le N$, then
\begin{gather*}
c_{k-1}=0,\quad b_{k+1}=0,\\
b_k+c_k=-1,\quad a_{k-1}=-b_k,\quad a_k=-c_k.
\end{gather*}
with the usual agreement that $c_1=b_{N+1}=0$. 
\end{enumerate}
This leaves only the possibility \eqref{eq:SkN} for $L$.

Conversely, consider a matrix $S_{k,N}(t)$ as in \eqref{eq:SkN} and rewrite it as 
\vspace{2\baselineskip}
$$
S_{k,N}=
\begin{bNiceMatrix}[first-col]
{\scriptstyle k-1} \{   & B & D & 0 \\
{\scriptstyle 2} \{     & 0 & C & 0 \\
{\scriptstyle N-k-1} \{ & 0 & E & F
\CodeAfter
  \OverBrace[yshift=1mm]{1-1}{1-1}{{\scriptstyle k-1}}
  \OverBrace[yshift=1mm]{1-2}{1-2}{{\scriptstyle 2}}
  \OverBrace[yshift=1mm]{1-3}{1-3}{{\scriptstyle N-k-1}}
\end{bNiceMatrix}\,.
$$ 
By a permutation similarity, $S_{k,N}$ can be brought to a block upper
triangular form, hence its eigenvalues are those of its diagonal blocks $B$, $C$, $F$, that is they are
$
\{0,1,\ldots,1\}.
$
\end{proof}


\begin{remark}
For $0<t<1$ and $2\le k\le N-2$, the network structure identified by $L$ in \eqref{eq:SkN} is that of a network having two adjacent leaders, namely nodes $k$ and $k+1$, and the coupling between them is bidirectional. The network graph has a source strongly connected component (SSC) consisting of the two leaders with two branches emanating from the SSC and no other branching. There are a total of $N$ edges. At the endpoint  $t=0$, for $k=1$, the network has only one leader and the graph is a directed path; similarly for  $t=1$ and $k=N-1$. 
Otherwise it is a spanning tree with one leader, two branches emanating from it and no farther branching. The graph has $N-1$ edges. These endpoint cases belong to the set of {\sl optimal networks} considered by Nishikawa and Motter in \cite[Theorem 7]{NM2006} and in \cite{NM2010}. 
\end{remark}


We now describe the Jordan block structure of the matrices $S_{k,N}$ in \eqref{eq:SkN} relative to the eigenvalue $1$.
\begin{theorem}\label{thm:Jordan_block_structure_trid}
Let $S_{k,N}(t)$ be as in Theorem~\ref{thm:classification_trid}. If $0<t<1$ and $2\le k\le N-2$, then $S_{k,N}(t)$ has exactly two Jordan blocks associated with the eigenvalue $1$, of sizes
$$
\max\{k,N-k\}
\qquad\text{and}\qquad
\min\{k,N-k\}-1.
$$
If $k=1$ or $k=N-1$, then $S_{k,N}(t)$ has only one Jordan block of size $N-1$ associated with the eigenvalue $1$.

The smallest possible size of the largest Jordan block of $S_{k,N}$ associated to the eigenvalue $1$ is
\begin{equation}\label{SmallestJordan}
	\left\lceil \frac{N}{2} \right\rceil.
\end{equation}
\end{theorem}

\begin{proof}
Let $0<t<1$ and $2\le k\le N-2$, and set
$$
A = A(t) = I_N-S_{k,N}(t),
$$
so it is enough to look at the Jordan structure of $A$ corresponding to the eigenvalue $0$. Observe that $A$ is similar via a permutation to
	
\begin{equation}\label{eq:3x3_block_matrix}
\tilde{A}=\bmat{
J_{k-1} & 0 & E_1\\
0 & J_{N-k-1} & F_1\\
0 & 0 & C
},
\end{equation}
where $J_j$ denotes the $j\times j$ Jordan block
$\smat{
0&1&0&\sdots&0\\
0&0&1&\sddots&\svdots\\
\svdots&\svdots&\sddots&\sddots&0\\
0&0&\sdots&0&1\\
0&0&\sdots&0&0
}$,
and
$C=C(t)=\smat{1-t&t\\ 1-t&t}$, $E_1=
\smat{0&0\\\svdots&\svdots\\0&0\\1&0} \in\R^{(k-1)\times 2}$,
$F_1=\smat{0&0\\\svdots&\svdots\\0&0\\0&1}\in\R^{(N-k-1)\times 2}$.
Next, look at the rank of the powers of $\tilde{A}$. 
For $m\ge1$, we have 
\begin{equation}\label{rankAtilde}
\begin{gathered}
\tilde A^m=\bmat{
J_{k-1}^m & 0 & E_m\\
0 & J_{N-k-1}^m & F_m\\
0 & 0 & C^m},\\[1ex]
\rank(\tilde A^m)
=\rank(J_{k-1}^m)+\rank(J_{N-k-1}^m)
+\rank\left(\bmat{E_m\\F_m\\C}\right).
\end{gathered}
\end{equation}
In \eqref{rankAtilde}, we have
$E_m=
\smat{
0&0\\
\svdots&\svdots\\
0&0\\
1&0\\
1-t&t\\
\svdots&\svdots\\
1-t&t
}\in\R^{(k-1)\times 2}$,
where the vector $\bmat{1&0}$ occupies row $k-m$ in $E_m$ for $1\le m\le k-1$, while $E_m$ has all rows equal to $\bmat{1-t&t}$ for $m\ge k$.
Similarly,
$F_m=
\smat{
0&0\\
\svdots&\svdots\\
0&0\\
0&1\\
1-t&t\\
\svdots&\svdots\\
1-t&t
}\in\R^{(N-k-1)\times 2}$,
where the vector $\bmat{0&1}$ occupies row $N-k-m$ in $F_m$ for $1\le m\le N-k-1$, while $F_m$ has all rows equal to $\bmat{1-t&t}$ for $m\ge N-k$.
Finally, for all $m\ge1$ we have
$C^m=C=
\smat{
1-t&t\\
1-t&t
}$.

Now it is easy to see that:
$$
\rank(J_{k-1}^m)=
\begin{cases}
k-1-m, & 1\le m\le k-2,\\
0, & m\ge k-1,
\end{cases}
$$
$$
\rank(J_{N-k-1}^m)=
\begin{cases}
N-k-1-m, & 1\le m\le N-k-2,\\
0, & m\ge N-k-1,
\end{cases}
$$
$$
\rank\left(\bmat{E_m\\F_m\\C}\right)
=
\begin{cases}
2, & 1\le m\le \max\{k,N-k\}-1,\\
1, & m\ge \max\{k,N-k\}.
\end{cases}
$$

Therefore, from \eqref{rankAtilde}
we have:
\begin{equation*}
\resizebox{0.94\columnwidth}{!}{$\displaystyle
\rank(\tilde A^m)=
\begin{cases}
N-2m, & 1\le m\le \min\{k,N-k\}-1,\\[1mm]
\max\{k,N-k\}-m+1,
& \min\{k,N-k\}\le m\le\max\{k,N-k\}-1,\\[1mm]
1, & m\ge\max\{k,N-k\}.
\end{cases}$}
\end{equation*}
It follows that $\tilde A$, and hence $A$, has exactly two Jordan blocks associated with the eigenvalue $0$, of sizes
$$
\max\{k,N-k\}
\qquad\text{and}\qquad
\min\{k-1,N-k-1\},
$$
as claimed.
The statement for $k=1$ or $k=N-1$ is immediate, since in these cases $S_{k,N}(t)$ is essentially already in Jordan  form, with only one block associated with the eigenvalue $1$, of size $N-1$.


The claim on the smallest possible size of the Jordan block is a simple verification.  By virtue of the above result, one must choose $k$ as close as possible to $N-k$, and this gives
\begin{itemize}
\item if $N$ is even, the optimal choice is $k=N/2$, and the size of the largest block is $N/2$;
\item if $N$ is odd, the optimal choices are $k=(N\pm1)/2$, and the size of the largest block is $(N+1)/2$. 
\end{itemize}
To sum up, the smallest possible size of the largest Jordan block of $S_{k,N}$ associated to the eigenvalue $1$ is indeed given by \eqref{SmallestJordan},
as claimed.
\end{proof}

\begin{remark}\label{rem:endpoint cases}
    The argument used in the proof of the previous theorem allows to cover the endpoint cases $t=0$ and $t=1$. The smallest possible size of the largest Jordan block of $S_{k,N}(0)$ or $S_{k,N}(1)$ associated to the eigenvalue $1$ is
	$\left\lfloor \frac{N}{2} \right\rfloor$.
\end{remark}

\begin{remark}\label{rem:sizeLargestJordanBlock_trid}
As already remarked, having a normalized spread of $0$ optimizes
the synchronizability of the network, but it is the size of the largest Jordan block of the Laplacian matrix that affects the duration of the transient: 
see Section \ref{Transient} and particularly \eqref{TransGrowth}, as well as the results in Section \ref{NumSim}.  Of course, this fact has been observed many times before; e.g., Fish and Sun in \cite{FishSun} experimentally show that networks with larger Jordan blocks (what they call {\emph sensitivity index}) generally take a longer amount of time to reach synchronization, if at all, and in \cite{FB2022} they define an alternative measure to the MSF, what they call Laplacian pseudospectral resilience (LPR), to assess synchronizability.  They further propose a stochastic algorithm aiming to obtain networks with a target sensitivity index, but they cannot a-priori prescribe the network structure.  In our present context of tridiagonal networks minimizing the value of the spread, and 
hence within the family of matrices $S_{k,N}$, 
our Theorem \ref{thm:Jordan_block_structure_trid} rigorously establishes that the smallest possible size of the largest Jordan block of $S_{k,N}$ associated to the eigenvalue $1$ is
$\left\lceil \frac{N}{2} \right\rceil$,
and this is the value guaranteeing the shortest transient within the class of tridiagonal Laplacians minimizing the spread. 
\end{remark}

\begin{remark}
Obviously, for $N\ge 4$, the tridiagonal structure of Theorems \ref{thm:classification_trid} and \ref{thm:Jordan_block_structure_trid} is that of a non-diagonalizable Laplacian.  Again, this by itself is not a surprise.  In an influential series of works, see \cite{NM, NM2006}, Nishikawa and Motter had provided ample evidence that maximally synchronizable networks are necessarily nondiagonalizable and they also argued that they need to have the structure of directed graphs. Our results establish rigorously that -for the class of tridiagonal networks, and in order to minimize both spread and transient length- it is best to have either a graph structure of two leaders,  each having essentially the same number of followers, or a graph structure with one leader and two separate branches with essentially the same number of followers. 
\end{remark}

%

\section{Bandwidth $b\ge 2$ Laplacians with $\lambda_2=\ldots=\lambda_N=1$}\label{BandLap}


It is natural to ask whether 
a characterization like that of Theorem \ref{thm:classification_trid} for tridiagonal Laplacians with zero normalized spread
can be extended to construct zero-normalized-spread Laplacians of arbitrary bandwidth $b\ge2$, and whether doing so provides any benefit with respect to the tridiagonal case. Here, we show that the answer to both questions is affirmative.

\begin{definition}\label{UpperBanded}
We say that a matrix $A=(a_{ij})$ is upper banded with bandwidth $b$ if $a_{ij}=0$ whenever either $i>j$ or $j-i>b$. We say that $A$ is lower banded with bandwidth $b$ if its transpose is upper banded with the same bandwidth.
\end{definition}

\begin{theorem}\label{thm:banded_Laplacians}
Let $b\ge2$ and let $L\in\R^{N\times N}$ be a Laplacian matrix. Assume that, for some $k=1,\ldots,N-b$, the matrix $L$ has the following structure:
\begin{enumerate}[label=\roman*)]
\item the first $k-1$ rows are upper banded with bandwidth $b$, and $L_{ii}=1$ for $i=1,\ldots,k-1$;
\item the last $N-k-b$ rows are lower banded with bandwidth $b$, and $L_{ii}=1$ for $i=k+b+1,\ldots,N$;
\item the remaining rows $k,\ldots,k+b$ form a $(b+1)\times(b+1)$ core block $C$, with no nonzero entries outside columns $k,\ldots,k+b$, and $C$ has spectrum $\{0,1,\ldots,1\}$.
\end{enumerate}
Then $L$ has bandwidth $b$ and spectrum $\{0,1,\ldots,1\}$.
\end{theorem}
\begin{proof}
Indicating with $e_1,\dots, e_N$, the columns of the identity matrix $I_N$,
the result follows easily by applying the permutation similarity $\tilde{L}=P^TLP$, with
\begin{equation}\label{eq:permutation_banded}
\resizebox{0.94\columnwidth}{!}{$\displaystyle
P=\bmat{
e_1&\ldots&e_{k-1}&e_N&e_{N-1}&\ldots&e_{k+b+1}&e_k&\ldots&e_{k+b}}. $}
\end{equation}
and realizing that $\tilde{L}$ is block upper triangular with two blocks, the first being upper triangular with diagonal all 
entries equal to $1$, and the second being the $(b+1)\times(b+1)$ core.
\end{proof}

\begin{example}\label{example:banded_Laplacians}
For instance, an $(N-3)$-parameter family of 
$N\times N$ Laplacians with $b=2$ from Theorem \ref{thm:banded_Laplacians} is given by
\begin{equation}\label{eq:bandwidth2_example}
\resizebox{0.94\columnwidth}{!}{$\displaystyle
L=\bmat{
1 & -\ell_1 & \ell_1-1 & 0 & 0 & 0 & 0 & 0 & 0\\
 & \ddots & \ddots & \ddots & 0 & 0 & 0 & 0 & 0\\
0 & 0 & 1 & -\ell_{k-1} & \ell_{k-1}-1 & 0 & 0 & 0 & 0\\
0 & 0 & 0 & \frac23 & -\frac13 & -\frac13 & 0 & 0 & 0\\
0 & 0 & 0 & -\frac13 & \frac23 & -\frac13 & 0 & 0 & 0\\
0 & 0 & 0 & -\frac13 & -\frac13 & \frac23 & 0 & 0 & 0\\
0 & 0 & 0 & 0 & \ell_{k+3}-1 & -\ell_{k+3} & 1 & 0 & 0\\
0 & 0 & 0 & 0 & 0 & \ddots & \ddots & \ddots & 0\\
0 & 0 & 0 & 0 & 0 & 0 & \ell_N-1 & -\ell_N & 1
}. $}
\end{equation}
where
$0\le \ell_i\le 1$, $i=1,\ldots,k-1,\ k+3,\ldots,N$,
and for simplicity the $3\times3$ core block has been chosen to be
$C=I_3-\frac13\bmat{1\\1\\1}\bmat{1&1&1}$,
though $C$ could have been any
$3\times 3$ Laplacian matrix with eigenvalues $\{0, 1, 1\}$.
\end{example}

\begin{remark}
Note that it is not true (unlike the tridiagonal case of Theorem \ref{thm:classification_trid}) that Theorem \ref{thm:banded_Laplacians} is an
``if and only if'' statement.  In fact, it is not true that every bandwidth-$b$ Laplacian with eigenvalues $\{0,1,\ldots,1\}$ must be of the form given by Theorem \ref{thm:banded_Laplacians}. E.g., the matrix
$L=
\smat{
\frac34 & -\frac58 & -\frac18 & 0\\
-\frac1{16} & \frac34 & -\frac12 & -\frac3{16}\\
-\frac3{16} & -\frac12 & \frac34 & -\frac1{16}\\
0 & -\frac18 & -\frac58 & \frac34
}
$
is a bandwidth-$2$ Laplacian with spectrum $\{0,1,1,1\}$, but it does not have the structure described by Theorem \ref{thm:banded_Laplacians}.
\end{remark}

\begin{remark}\label{RemBand}
Similarly to Theorem \ref{thm:Jordan_block_structure_trid},
for the matrices constructed in Theorem \ref{thm:banded_Laplacians}, the Jordan structure associated with the eigenvalue $1$ is governed by the upper/lower banded blocks and by the core. Assume that
$2\le k\le N-b-1$, so that neither upper/lower block is empty. 
If the first superdiagonal entries in the upper banded block and the first subdiagonal entries in the lower banded block are all nonzero, then these blocks contribute one Jordan block each, of sizes
$$
k-1,
\qquad
N-k-b,
$$
respectively. For large $N\gg b$, the contribution of the core to the size of the largest Jordan block is negligible, and this size is given by
$\max\{k-1,\;N-k-b\}$.
This quantity is minimized by choosing $k-1$ as close as possible to $N-k-b$, which gives
$$\left\lceil\frac{N-b-1}{2}\right\rceil,
$$
which, again for $N\gg b$, returns the same value $\left\lceil\frac{N}{2}\right\rceil$ we found in 
\eqref{SmallestJordan}
for the tridiagonal case.
\end{remark}

Remark \ref{RemBand} shows that, in general, increasing the bandwidth $b$ does not lead to a smaller largest Jordan block. To achieve this, one must make a suitable choice of the upper and lower banded portions of $L$ in Theorem \ref{thm:banded_Laplacians}.

\begin{theorem}\label{thm:optimal_banded_Laplacians}
Let $L\in\RNN$ be a bandwidth $b\ge2$ Laplacian satisfying the assumptions in \emph{i)}, \emph{ii)}, \emph{iii)} of Theorem \ref{thm:banded_Laplacians}.
Assume moreover that the core block $C$ is diagonalizable and has no zero rows, and that
\begin{enumerate}[label=\roman*)]
\item $L_{i,i+1}=\ldots=L_{i,i+b-1}=0$, $i=1,\ldots,k-1$;
\item $L_{i,i-1}=\ldots=L_{i,i-b+1}=0$, $i=k+b+1,\ldots,N$.
\end{enumerate}
Then, the size of the largest Jordan block of $L$ associated with the eigenvalue $1$ is
$$
\max\left\{
\left\lceil\frac{k+b-1}{b}\right\rceil,
\left\lceil\frac{N-k}{b}\right\rceil
\right\}.
$$
\end{theorem}
\begin{proof}
The proof follows closely the proof of Theorem \ref{thm:Jordan_block_structure_trid}, and therefore we only sketch it. Through the permutation similarity in \eqref{eq:permutation_banded}, $A=I_N-L$ is similar to a block upper triangular matrix $\widetilde A$, 
whose diagonal blocks are two nilpotent blocks and the block $I-C$, where $C$ is the core. The extra zero conditions \emph{i)}, \emph{ii)} cause each nilpotent block to lose rank by $b$ at each power of $\widetilde A$. Hence, as $m$ increases by one, the rank of $\widetilde A^m$ decreases by $2b$ while both nilpotent blocks are nonzero, then by $b$ once only one of them is nonzero, and finally by one before stabilizing, due to the hypothesis on the core $C$ (diagonalizability and no zero rows).
\end{proof}

\begin{example}\label{example:S492}
For instance, choosing $N=9$, $k=4$, and $\ell_i=0$ for $i=1,2,3,7,8,9$ in Example \ref{example:banded_Laplacians}, we obtain the following bandwidth-$2$ Laplacian:
\begin{equation}\label{eq:Skn2}
\resizebox{0.8\columnwidth}{!}{$\displaystyle
L=\bmat{1 & 0 & -1 & 0 & 0 & 0 & 0 & 0 & 0\\
0 & 1 & 0 & -1 & 0 & 0 & 0 & 0 & 0\\
0 & 0 & 1 & 0 & -1 & 0 & 0 & 0 & 0\\
0 & 0 & 0 & \frac23 & -\frac13 & -\frac13 & 0 & 0 & 0\\
0 & 0 & 0 & -\frac13 & \frac23 & -\frac13 & 0 & 0 & 0\\
0 & 0 & 0 & -\frac13 & -\frac13 & \frac23 & 0 & 0 & 0\\
0 & 0 & 0 & 0 & -1 & 0 & 1 & 0 & 0\\
0 & 0 & 0 & 0 & 0 & -1 & 0 & 1 & 0\\
0 & 0 & 0 & 0 & 0 & 0 & -1 & 0 & 1
}. $}
\end{equation}
In this case of \eqref{eq:Skn2}, the largest Jordan block associated with the eigenvalue $1$ has size
$$
\max\left\{
\left\lceil\frac{5}{2}\right\rceil,
\left\lceil\frac{5}{2}\right\rceil,
1
\right\}=3.
$$
\end{example}

In Theorem \ref{thm:optimal_banded_Laplacians}, the assumption that the core is diagonalizable  significantly simplifies the statement of the theorem without limiting its relevance. It is also consistent with the goal of minimizing the size of the largest Jordan block. For this reason, we define the following class of Laplacians.  
\begin{definition}[bandwidth-$b$ Laplacian]\label{def:bandwith-b-Laplacian}
We say that a Laplacian $L$ as in Theorem \ref{thm:optimal_banded_Laplacians} is a bandwidth-$b$ Laplacian, written as $S_{k,N}^b$, if 
\begin{equation}\label{eq:SkNb}
C = I_{b+1}-\frac{1}{b+1}
\bmat{1\\ \vdots\\ 1}\bmat{1&\cdots&1}.
\end{equation}
\end{definition}

\begin{remark} For example, 
$L$ in Example \ref{example:S492} is $L=S_{4,9}^2$.
\end{remark}

We now show that the matrices $S_{k,N}^b$ are the best one can get in terms of Jordan structure to minimize the length of the transient.

\begin{theorem}\label{thm:lower_bound_Jordan_blocks_banded}
Let $L\in\R^{N\times N}$ be a  bandwidth $b\ge2$ Laplacian matrix, hence with $\sigma(L)=\{0,1,\ldots,1\}$.
Then, the size of the largest Jordan block of $L$ associated with the eigenvalue $1$ cannot be smaller than
$$
\left\lceil \frac{N-1}{2b} \right\rceil.
$$
If, in addition, $L$ does not have any zero row, then the size of the largest Jordan block of $L$ associated with the eigenvalue $1$ cannot be smaller than
$$
\left\lceil \frac{N}{2b} \right\rceil.
$$
\end{theorem}
\begin{proof}
Set
$A=I_N-L$, so that 
$
\sigma(A)=\{0,\ldots,0,1\}$, and 
$A\ones=\ones$,
where $\ones$ is the vector of all 1's. 
Let $m\ge1$ be the size of the largest Jordan block of $L$ associated with the eigenvalue $1$, and so
$
\operatorname{rank}(A^m)=1$.  Then, 
there exists a vector $w\in\mathbb{R}^N$, with
$\sum_{j=1}^N w_j=1$, such that
$A^m=\ones w^T$.

Since $L$ has bandwidth $b$, also $A=I_N-L$ has bandwidth $b$, and therefore $A^m$ has bandwidth $mb$. On the other hand, all rows of $A^m$ are equal to $w^T$. Hence, in order for $A^m$ to have bandwidth $mb$ we must have
$$
w_j=0\quad\text{if}\quad j<N-mb \quad\text{or}\quad j>mb+1.
$$
Therefore $w_j$ can be nonzero only for
$
j\in\{N-mb,\ldots,mb+1\}$.
But since $w\neq0$, this set must be nonempty, which gives
$mb+1\ge N-mb$, and this implies
$$
m\ge \left\lceil\frac{N-1}{2b}\right\rceil.
$$
Now, suppose that no row of $L$ is zero. We now show that the set
$$
\{N-mb,\ldots,mb+1\}
$$
cannot contain exactly one integer. Indeed, if this happened, then $w=e_j$ for some $j$, and hence
$
A^m=\ones e_j^T$, and in particular,
$
e_j^T A^m=e_j^T$.  But, 
since $A,A^2,\ldots,A^m$, all have the same left eigenspace associated with the eigenvalue $1$, it follows that
$e_j^T A=e_j^T$.  
Equivalently, the $j$-th row of $A=I_N-L$ is $e_j^T$, and therefore the $j$-th row of $L$ is zero, but this is not possible. Hence the set $\{N-mb,\ldots,mb+1\}$ must contain at least two integers. Therefore
$mb\ge N-mb$,
that is
$
m\ge \left\lceil\frac{N}{2b}\right\rceil$ as claimed.
\end{proof}


\begin{remark}
Theorem \ref{thm:lower_bound_Jordan_blocks_banded} shows that, with the appropriate choice of $k$, the matrices $S_{k,N}^b$, are optimal in the following sense: choosing $k-1$ as close as possible to $N-k-b$, that is placing the core as centrally as possible, $S_{k,N}^b$ attains the minimal possible size of the largest Jordan block associated with the eigenvalue $1$ among all $N\times N$ bandwidth-$b$ Laplacians of Definition \ref{def:bandwith-b-Laplacian}.
\end{remark}

\section{Numerical experiments}\label{NumSim}
In this section, we provide numerical simulations of networks 
for different choices of the Laplacians. 
We will use the following Laplacians: 
\begin{itemize}
\item[$L_{(i)}$] $=\lambda S_{1,N}(0)$, where $S_{1,N}(0)$ is the tridiagonal Laplacian matrix in Theorem \ref{thm:classification_trid} with $t=0$. The first agent is the unique leader while the other agents are followers. The scalar $\lambda$ is chosen so that the MSF is negative.  
\item[$L_{(ii)}$] $=\lambda S_{k,N}(\frac 12)$, where $S_{k,N}(\frac 12)$ is the tridiagonal Laplacian matrix in Theorem \ref{thm:classification_trid} with $t=\frac 12$ and symmetric core $C=I_2 -\frac 12 \begin{pmatrix} 1 & 1 \\ 1 & 1 \end{pmatrix}$ in rows and columns $k$, $k+1$. 
\item[$L_{(iii)}$] $=\lambda S_{k,N}^b$, for $b=2,3$, where $S_{k,N}^b$  is the Laplacian in Theorem \ref{thm:optimal_banded_Laplacians}. We take $\ell_i=0$ for all $i$ and choose the core $C$ as in \eqref{eq:SkNb}.
\item[$L_{(iv)}$] $=T$ as derived in \cite[Algorithm 2, Theorem 2.4]{DiElPu1}. Here, the matrix $T$ is 
tridiagonal and irreducibile so that is has $N$ simple eigenvalues 
$0=\lambda_1<\lambda_2<\lambda_3<\ldots<\lambda_N$, with $\lambda_2,\dots, \lambda_N$ chosen in an interval where the MSF is  negative. 
\end{itemize}

For each problem, and unless otherwise stated, we compute the MSF with the continuous QR method; see \cite{DV2002}. Together with the MSF, we also compute the greatest interval of the Exponential Dichotomy spectrum; see \cite{DV2002}. When this is negative, the system is uniformly asymptotically stable.  
For integrating
\eqref{Network2}, we use a fixed stepsize Runge Kutta method of order 4 and stepsize $h=2^{-6}$, unless otherwise specified. 

In order to measure the distance of solutions from the synchronous space, we use 
\begin{equation*}
\begin{aligned}
d(t)&=\left\|\left(I_{nN}-\frac1N
(\e\e^T\otimes I_n)\right)\x(t)\right\|_\infty\\
& = \max_{1\le i\le N}\|\x_i(t)-\frac{1}{N}\sum_{j=1}^N\x_j(t)\|_\infty \\
\e&=\smat{1\\\svdots\\1}\in\R^N, \\
\end{aligned}
\end{equation*}
that is we look at the 
sup-norm of the projection of the component of the solution into $\S^\perp$, at all $t$.  
 
\subsection{Lorenz system}\label{LorSim}

We consider a network of $N$ agents that satisfy the Lorenz equations
\begin{equation}\label{eq:lorenz}
\dot{\bf y} =
\begin{pmatrix}
10(y_2-y_1) \\
y_1(28-y_3)-y_2 \\
y_1y_2-2y_3
\end{pmatrix}.
\end{equation} 
In \eqref{Network2}, for the componentwise interaction, we take 
the matrix 
$E=\begin{pmatrix} 0 & 1 & 0 \\ 0 & 0 & 0 \\ 0 & 0 & 0 \end{pmatrix}$, Laplacian matrix any of $L_{(i)}, \,\ \ldots L_{(iv)}$ above, and simulate the network for different values of $k$, $N$, $b$ and $\lambda$. Our simulations clearly show how the transient length, in the cases $L_{(i)}, L_{(ii)}, L_{(iii)}$, depends on the size of the largest Jordan block and on the asymptotic convergence rate. 
In the case $L_{(iv)}$, which is diagonalizable, the transient length is impacted by the condition number of the eigenvector matrix of $T$. 
Interestingly, choosing $L_{(iv)}$ gives a transient length similar to, or shorter than, the case $L_{(i)}$, in spite of the fact that the network associated to $L_{(i)}$ is optimal insofar as the eigenvalues' spread is concerned, while the network associated to $T$ is not optimal.

The MSF for \eqref{eq:lorenz} has been computed many times before, 
(see \cite[Figure 2]{Huang}). 
We remark that the MSF does not depend 
on $L$. For our choice of $E$, the MSF is stable and it is negative in an open interval containing  $[4.25,\,22.5]$. In the same interval, the exponential dichotomy spectrum is contained in the negative real half-line, and hence the linearized system is uniformly exponentially stable.
For our simulations we take random initial conditions in a neighborhood of radius $0.1$ of a synchronous initial condition on the Lorenz attractor.
We use two values of $\lambda$, namely $\lambda=5$ and $\lambda=12$. The corresponding values of the MSF are $\Lambda(5) \simeq -0.14$ and $\Lambda(12) \simeq -2.34$. 
When the Laplacian is equal to $T$ in $L_{(iv)}$, we choose $\lambda_2, \ldots ,\lambda_N$, equal to Chebyschev points of the first kind translated and scaled in the interval $[\lambda -0.5 , \,\ \lambda+0.5]$.
For $\lambda=12$, we couple $N=512$ and $N=1024$ agents. 
In Figure \ref{Lorenz_fig} we plot the distance from the synchronous space for the different Laplacians. The dependence of the transient time on the sensitivity index of $\lambda$ is evident in the plots. 
\begin{figure}
    \centering
    \includegraphics[width=\columnwidth]{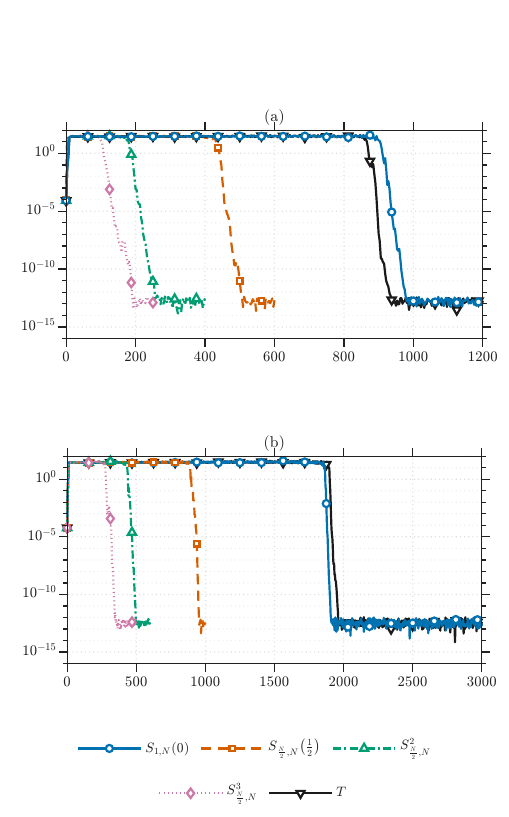}
\caption{Lorenz system \eqref{eq:lorenz}. Plot with Laplacian matrices $T$, $\lambda S_{1,N}(0)$, 
$\lambda S_{\frac N2 ,N}(\frac 12)$, $\lambda S_{\frac N2 N}^{2}$ and 
$\lambda S_{\frac N2 N}^{3}$  for $\lambda=12$.  
 Panel (a): $N=512$. Panel (b): $N=1024$. }
\label{Lorenz_fig} 
\end{figure}
 
For $\lambda=5$ the convergence rate is much slower, hence in 
Figure \ref{Lorenz_fig2} we consider only $N=256$ agents. On the left we plot the distance from the synchronous 
space for networks with Laplacian matrices $S_{1,N}(0)$, $S_{\frac N2, N}(\frac 12)$ and $S_{\frac N4 ,N}(\frac 12)$. 
In the plot, we can observe the linear dependence of the transient length on the dimension of the greatest Jordan block.
On the right, for the Laplacian matrix $S_{\frac N2,N}(\frac 12)$, we compare the two different convergence rates 
for $\lambda=12$ and $\lambda=5$, making it evident how 
the transient length depends on the convergence rate; see \eqref{TransGrowth}. 
\begin{figure}[!t]
\centering
\includegraphics[width=\columnwidth]{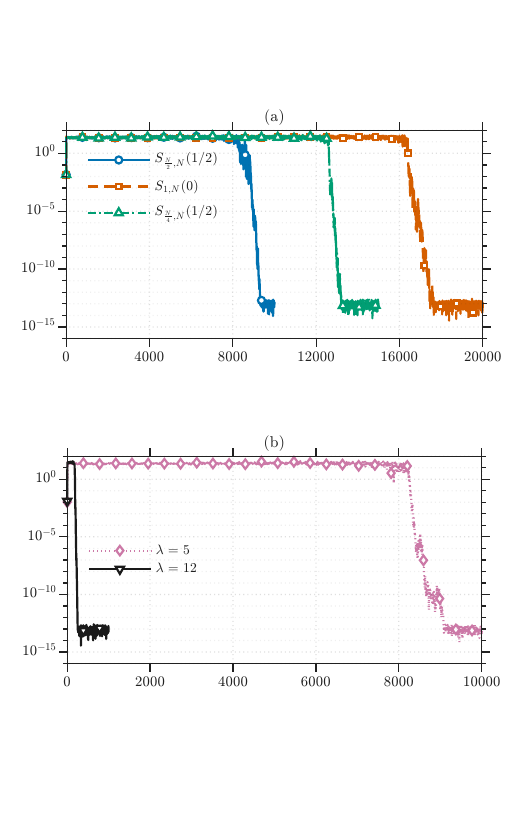}
\caption{Lorenz system \eqref{eq:lorenz} with $N=256$ agents. Panel (a): Laplacian matrices $S_{1,N}(0)$, $S_{\frac N2,N}(\frac12)$, and $S_{\frac N4,N}(\frac12)$ for $\lambda=5$. Panel (b): $S_{\frac N2,N}(\frac12)$ for $\lambda=12$ and $\lambda=5$.}
\label{Lorenz_fig2}
\end{figure}

\subsection{Lorenz '96}
We couple $N$ agents, each satisfying the Lorenz '96 equations
\begin{equation}\label{eq:Lorenz96}
\begin{aligned}
\frac{d y_i}{dt}&=(y_{i+1}-y_{i-2})y_{i-1}-y_i+F,\\
i&=0,\dots,n-1\pmod n.
\end{aligned}
\end{equation} 
and we select 
$F=8$ and $n=12$ and $n=40$.  For both
$n=12$ and $n=40$ the system has a chaotic attractor with respectively three and thirteen positive Lyapunov exponents. There is more than one expanding direction along the chaotic orbit and we want to test convergence to the synchronous space in this case.  
We computed the Lyapunov exponents of the network for different choices of the componentwise interaction matrix $E$ in \eqref{Network2}, and in order to 
find intervals giving negative MSF in \eqref{SynchR}, we choose $E$ as follows: $E_{ii}=1$, for $i$ odd, $E_{ij}=0$ otherwise. In Figure \ref{MSF_Lorenz96} we plot in function of $\eta$ both the greatest Lyapunov exponent of \eqref{Linear} and the rightmost interval of the exponential dichotomy spectrum.
For $n=12$, the Lyapunov exponent is negative  
in an open interval containing $[4.1, \,\ 32.9]$, while for $n=40$, it is negative in an open interval containing $[5.1, \,\ 22.5 ]$. 
\begin{figure}[!t]
\centering
\includegraphics[width=\columnwidth]{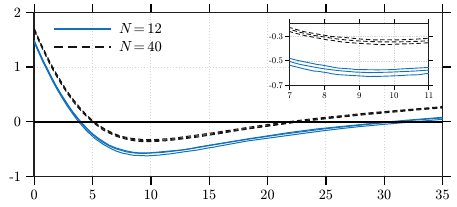}
\caption{Lorenz '96 system \eqref{eq:Lorenz96}. Largest Lyapunov exponent of \eqref{Linear} and rightmost interval of the exponential dichotomy spectrum interval as functions of $\eta$. The upper and lower curves represent the endpoints of the exponential dichotomy spectral interval, while the middle curve is the MSF. }
\label{MSF_Lorenz96}
\end{figure}
Given the size of the single agent, we simulate a network of $N=48$ agents for $n=12$, and a network with $N=24$ for $n=40$. We use the Laplacians in $L_{(i)}$, \ldots, $L_{(iv)}$, with $\lambda=9$. For $\eta=9$, in the case $n=12$, the Lyapunov exponent is $\Lambda(\eta) \simeq -0.58$, while for $n=40$,  
$\Lambda(\eta) \simeq -0.33$. 
For $\eta=9$, both linearized systems are uniformly asymptotically stable.  
When we choose $T$ as in $L_{(iv)}$ the eigenvalues are Chebychev points of first kind translated in $[8.5, \,\ 9.5]$. 
We plot the distance from the synchronous subspace in Figure \ref{Lorenz96}. From the two plots, we see that after an initial 
transient, the distance from the synchronous space decreases with the speed of the greatest Lyapunov exponent, a fact highlighted by the plot of $e^{\Lambda(\eta)t}$ (solid black line in the plots).
The oscillatory decay of the distance to the synchronous subspace is probably due to
the existence of several expanding directions.
\begin{figure}[!t]
\centering
\includegraphics[width=\columnwidth]{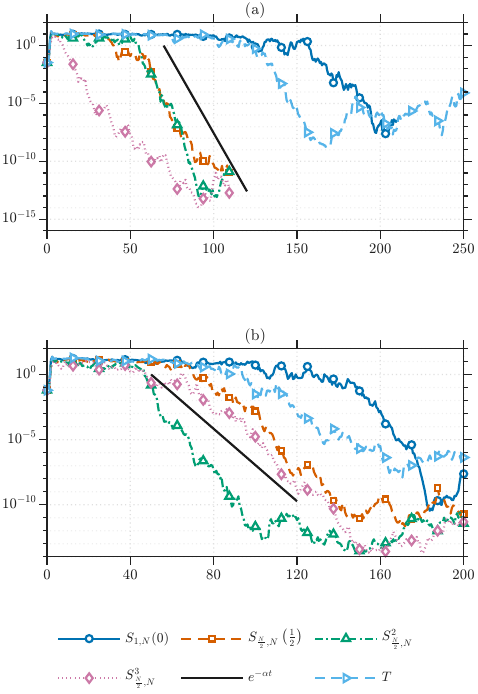}
\caption{Lorenz '96 system \eqref{eq:Lorenz96} for different choices of the Laplacian matrix. The exponential convergence compares well with the largest Lyapunov exponent. Panel (a): $N=48$, $n=12$, $\alpha=0.58$. Panel (b): $N=24$, $n=40$, $\alpha=0.33$.}
\label{Lorenz96}
\end{figure}


\subsection{Non smooth mechanical oscillator}\label{sec:nonsmooth_mechanical}

This last example is that of a network of planar non-smooth mechanical oscillators with dry friction. The single agent obeys the piecewise smooth system 
\begin{equation}\label{eq:nonsmooth-mech-agent}
\begin{aligned}
\dot \y&=\begin{cases}
\f^+(\y),&y_2>v,\\
\f^-(\y),&y_2<v,
\end{cases}\\[1mm]
\f^+(\y)&=\bmat{
y_2\\[1mm]
-y_1-\dfrac{1}{1+3(y_2-v)}},\\[1mm]
\f^-(\y)&=\bmat{
y_2\\[1mm]
-y_1+\dfrac{1}{1-3(y_2-v)}}.
\end{aligned}
\end{equation}
with parameter $v=0.15$.

%
On the discontinuity hyperplane $\Sigma=\{y\in\R^2:\ y_2-v=0\}$, we use Filippov vector field; so doing, the single agent has a periodic orbit with a sliding segment on $\Sigma$.  As in \cite{DiEl2023}, we compute the MSF using the
non-smooth single-agent equation while for direct simulations of networks with
several agents we integrate the regularized system \eqref{eq:nonsmooth-mech-regularized} below (see \cite{DiElLo24}):
\begin{equation}\label{eq:nonsmooth-mech-regularized}
\begin{aligned}
\f_\varepsilon(\y)={}&\alpha_\varepsilon(y_2-v)\f^+(\y)\\
&+\bigl(1-\alpha_\varepsilon(y_2-v)\bigr)\f^-(\y).
\end{aligned}
\end{equation}
where
\begin{equation}\label{eq:nonsmooth-mech-alpha}
    \alpha_\varepsilon(z)=
    \begin{cases}
        1, & z\ge \varepsilon,\\
        0, & z\le -\varepsilon,\\
        \dfrac12+\dfrac{z}{4\varepsilon}\left(3-\dfrac{z^2}{\varepsilon^2}\right),
        & |z|<\varepsilon,
    \end{cases}
\end{equation}
$z=y_2-v$. We use the coupling matrix $E=\begin{bmatrix}0&0\\1&0\end{bmatrix}$ and
the resulting network \eqref{Network} is
\begin{equation}\label{eq:nonsmooth-mech-network}
\begin{aligned}
\dot\x_i=&\f_\varepsilon(\x_i)
+\sum_{j=1}^N a_{ij}E(\x_j-\x_i),\\
&i=1,\ldots,N.
\end{aligned}
\end{equation}
Since the single agent has a periodic orbit, for this example it is easier to compute the Floquet multiplier of largest modulus (we will refer to it as the Floquet mutliplier below) instead of the largest Lyapunov exponent. 

For the chosen coupling matrix $E$, we compute the Floquet multiplier from the
non-smooth variational problem along the periodic orbit of the single agent following
\cite{DiElLo24}. The values of $\eta$ for which the multiplier lies in $(-1,1)$
predict transverse stability of the synchronous periodic orbit. In this example there
is a union of disjoint intervals on which the Floquet multiplier has modulus smaller
than $1$, i.e., on which the MSF is negative; see Fig.~\ref{fig:GalvanettoE1}.
In the computations below we take $L=\lambda S_{k,N}^{b}$ with $\lambda=2.1$,
which lies inside an interval giving negative MSF.
\begin{figure}[!t]
\centering
\includegraphics[width=\columnwidth]{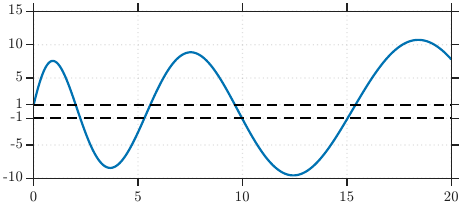}
\caption{Floquet multiplier for the non-smooth mechanical oscillator \eqref{eq:nonsmooth-mech-agent} as a function of $\eta$.}
\label{fig:GalvanettoE1}
\end{figure}

The network integration is performed with the regularized vector field \eqref{eq:nonsmooth-mech-regularized}, with the integrator {\tt ode45} from {\tt Matlab} using absolute and relative error tolerances of $10^{-12}$.  Initial values are chosen at distance $\rho$ from the sliding segment of the periodic orbit of the single agent.

We fix $k=N/2$ and compare $b=1,2,3$, for $N=32$ and regularization parameter $\epsilon=0.1$. 
As we can see from Figure \ref{fig:GalvaN32kN2}, for larger bandwidths of the Laplacian, the numerical solution converges to the synchronous subspace for larger values of $\rho$, indicating that for this example the basin of attraction of the synchronous orbit is larger for larger bands. We also consider the coupling matrix $T$ as in $L_{iv}$ with distinct eigenvalues chosen in the following union
$[2, \,\ 2.2] \cup [5.3, \,\ 5.5] \cup [9.6, \,\ 9.9] \cup [15, \,\ 15.3]$. The MSF is negative for $\eta$ in the chosen intervals.   
As we can see from the plot, in our simulations we can take the radius of convergence for Laplacian equal to $T$ larger than the one for Laplacians $S_{N/2,N}$ and 
$S_{N/2,N}^2$. 
\begin{figure}[ht]
    \centering
\includegraphics[width=\columnwidth]{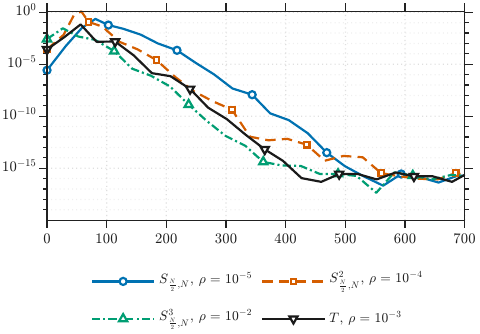}
\caption{Regularized network \eqref{eq:nonsmooth-mech-network}, $N=32$, simulations with matrices 
$S_{\frac N2, N}^b$, for $b=1,2,3$ and matrix $T$ with distinct eigenvalues. 
The regularization parameter is $\epsilon=0.1$.}
\label{fig:GalvaN32kN2}
\end{figure}

\section{Conclusions}\label{Concl}
In this work we gave rigorous results of optimality of network structure for the cases of tridiagonal, and banded, Laplacian matrices, so to enhance  synchronizability.  Our goal has been to  find suitable Laplacians whose Jordan canonical form minimizes both the spread and the transient length. Our results provide a firm mathematical foundation for optimizing the connections in a network of a given (tridiagonal or banded) structure.  Extensive numerical experiments validate our results.

\end{document}